\documentclass[12pt]{article}
\usepackage{amsmath,mathtools,amsfonts,amssymb,amsthm,bbm,
            graphicx,booktabs,enumerate}            
\usepackage[linesnumbered,ruled,lined]{algorithm2e}
\usepackage[active]{srcltx}
\usepackage[utf8]{inputenc}

\DeclareMathOperator{\argmin}{argmin}
\DeclareMathOperator{\aff}{aff}
\DeclareMathOperator{\nnz}{nnz}

\begin{document}
\newtheorem{oberklasse}{OberKlasse}
\newtheorem{definition}[oberklasse]{Definition}
\newtheorem{lemma}[oberklasse]{Lemma}
\newtheorem{proposition}[oberklasse]{Proposition}
\newtheorem{theorem}[oberklasse]{Theorem}
\newtheorem{corollary}[oberklasse]{Corollary}
\newtheorem{remark}[oberklasse]{Remark}
\newtheorem{example}[oberklasse]{Example}

\newcommand{\R}{\mathbbm{R}}
\newcommand{\N}{\mathbbm{N}}
\newcommand{\Z}{\mathbbm{Z}}
\newcommand{\C}{\mathbbm{C}}
\newcommand{\mc}{\mathcal}
\newcommand{\eps}{\varepsilon}
\renewcommand{\phi}{\varphi}
\newcommand{\cA}{{\mathcal A}}

\allowdisplaybreaks[1] 

\title{Generalized Gearhart-Koshy acceleration for the Kaczmarz method}
\author{J.\ Rieger}
\date{\today}
\maketitle

\begin{abstract}
The Kaczmarz method is an iterative numerical method for solving 
large and sparse rectangular systems of linear equations.
Gearhart, Koshy and Tam have developed an acceleration technique 
for the Kaczmarz method that minimizes the distance to the desired solution
in the direction of a full Kaczmarz step.

The present paper generalizes this technique to an acceleration scheme that 
minimizes the Euclidean norm error over an affine subspace spanned by a number 
of previous iterates and one additional cycle of the Kaczmarz method. 
The key challenge is to find a formulation in which all parameters of the
least-squares problem defining the unique minimizer are known,
and to solve this problem efficiently.

A numerical experiment demonstrates that the proposed affine search has the 
potential to clearly outperform the Kaczmarz and the randomized Kaczmarz 
methods with and without the Gearhart-Koshy/Tam line-search.
\end{abstract}



\section{Introduction}

The Kaczmarz method for solving systems of linear equations was initially described 
and analyzed in \cite{Kaczmarz}.
It was later rediscovered in the paper \cite{Gordon:70} in the context of computerized
tomography problems, where it was used with great success.
Being a row-action method, it tends to behave well when applied to large and sparse
rectangular linear systems, see \cite{Censor}.

The performance of the Kaczmarz method depends on the fixed order in which the 
method cycles through the rows of the linear system.
The randomized Kaczmarz method avoids the selection of a disadvantageous
order by selecting the rows at random.
It was proved in \cite{Strohmer} that this approach yields linear convergence
in expectation with a rate corresponding to the geometry of the problem. 

\medskip

Recently, there has been a strong emphasis on the development of acceleration schemes
for the randomized Kaczmarz method.
Some are based on splittings of the set of rows into a priori known well-conditioned
blocks, see \cite{Necoara} and \cite{Needell}, while others are based on Motzkin
acceleration see \cite{Motzkin}. 
The latter method selects the next row of the linear system corresponding 
to the largest component of the current residual instead of cycling through the rows 
in a given order.
Novel probabilistic variants of this approach select the next row with 
a probability distribution proportional to or otherwise derived from the current
residual, see \cite{Bai}, \cite{Haddock}, \cite{Haddock:Needell}, \cite{Steinerberger}
and the references therein.
In a sense, these methods are greedy algorithms that aim at decreasing the
residual as fast as possible.

\medskip

The line-search proposed in \cite{Gearhart} by Gearhart and Koshy for homogeneous
and recently in \cite{Tam} by Tam for inhomogeneous linear systems in the 
context of the deterministic Kaczmarz method is pursuing a greedy 
strategy that is diametrically opposed to Motzkin acceleration:
It uses one full cycle of the Kaczmarz method as a search direction and minimizes
the Euclidean norm error (instead of the residual) over the corresponding line.
This is achieved using only information that is explicitly known at runtime, which
means that this acceleration is computationally inexpensive.

Example 3.24 in \cite{Bauschke} shows that in pathological situations, the Kaczmarz
method with this line-search can be slower than the plain Kaczmarz method, while a
straight-forward modification of the convergence proof in \cite{Bregman} reveals that
it is necessarily convergent.

\medskip

The aim of this paper is to generalize the Gearhart-Koshy line-search 
to an acceleration scheme that minimizes the Euclidean norm error over an 
affine subspace spanned by a number of previous iterates and one additional 
cycle of the Kaczmarz method. 
This acceleration strategy is not limited to the deterministic Kaczmarz method,
but can be applied to the randomized Kaczmarz method as well.

The key challenge is to find a formulation in which all parameters of the
least-squares problem defining the unique minimizer are known,
and to solve this problem efficiently.
It turns out that this is possible in linear time because of the particular
structure of the problem.
A numerical experiment provided in the final section of the paper demonstrates 
that the proposed affine search has the potential to clearly outperform 
the Kaczmarz and the randomized Kaczmarz methods with and without the
Gearhart-Koshy line-search.

\medskip 

Finally, we would like to point out that the proposed method does not compete with the 
above-mentioned accelerations based on splittings and Motzkin acceleration, 
but it can in principle be applied to further enhance these and other methods 
based on successive projections.

\section{Preliminaries}

Throughout this paper, we consider a matrix
\[A=(a_1,\ldots,a_m)^T\in\R^{m\times n}\]
with rows $a_j\in\R^n\setminus\{0\}$ and a vector $b\in\mc{R}(A)$ in the range of $A$, 
and we consider the projectors 
\[P_j:\R^n\to\R^n,\quad
P_j(x):=(I-\frac{a_ja_j^T}{\|a_j\|^2})x+\frac{b_j}{\|a_j\|^2}a_j,\quad j=1,\ldots,m,\]
which project any point $x\in\R^n$ to the affine subspaces
\[H_j:=\{z\in\R^n:a_j^Tz=b_j\},\quad j=1,\ldots,m.\]
Their compositions
\[P(x):=(P_m\circ\ldots\circ P_1)(x)\]
constitute a full cycle of the Kaczmarz method.
It is well-known that for any $x_0\in\R^n$, we have 
$\lim_{k\to\infty}P^k(x_0)\in A^{-1}b$, 
see e.g.\ \cite{Tanabe}.

\medskip

When determining the computational complexity of the Kacz\-marz method and the
accelerated variants discussed in this paper, we will denote the number of nonzero
elements of the matrix $A$ by $\nnz(A)>0$.
As the matrix $A$ is large in typical applications, we will assume that scalar
quantities such as the norms $\|a_j\|^2$ can be stored, but not the normalized rows
$a_j/\|a_j\|$.
In this situation, we can carry out one Kaczmarz cycle with $4\nnz(A)+m$ flops.

\medskip

\begin{algorithm}\label{plain:Kaczmarz}
\caption{Kaczmarz method, 
originally proposed in \cite{Kaczmarz}\\
(complexity: $4\nnz(A)+m$ flops per cycle)}
\KwIn{
    $A\in\R^{m\times n}$,
    $b\in\R^m$,
    $x_0\in\R^n$   
}
\For{$k=0$ \KwTo $\infty$}{
    $x_{k+1}\gets P(x_k)$\;
}
\end{algorithm}

Inspired by the papers \cite{Gearhart} and \cite{Tam}, we wish to explore 
how the residual
\begin{equation}\label{res:def}
r(x):=\begin{pmatrix}
(a_1^Tx-b_1)/\|a_1\|\\
(a_2^TP_1(x)-b_2)/\|a_2\|\\
\vdots\\
(a_m^TP_{m-1}\circ\ldots\circ P_1(x)-b_m)/\|a_m\|
\end{pmatrix}
\end{equation}
can be used to speed up the Kaczmarz iteration.
Note that the quantities required for its computation are explicitly calculated 
in a cycle of the Kaczmarz method. 

\medskip

We begin by exploring the properties of the residual $r$,
which measures the reduction of the square distance to any solution
of the linear system in one Kaczmarz cycle and encodes information
on the angle between the vectors $x^*-x$ and $P(x)-x$.

\begin{lemma} \label{improvement}
Let $x^*\in A^{-1}b$, and let $x\in\R^n$ be arbitrary.
Then we have 
\begin{align}
&\|r(x)\|^2+\|P(x)-x^*\|^2=\|x-x^*\|^2,\label{imp:1}\\
&\|r(x)\|^2+\|P(x)-x\|^2=2(x-x^*)^T(x-P(x)).\label{imp:2}
\end{align}
\end{lemma}

\begin{proof}
Since
\begin{align*}
\langle P_j(x)-x,P_j(x)-x^*\rangle
&=\langle\frac{b_j-a_j^Tx}{\|a_j\|^2}a_j,x+\frac{b_j-a_j^Tx}{\|a_j\|^2}a_j-x^*\rangle\\
&=\frac{b_j-a_j^Tx}{\|a_j\|^2}\Big(a_j^Tx+(b_j-a_j^Tx)-b_j\Big)=0,
\end{align*}
we may use the Pythagorean theorem to compute
\begin{align*}
\|x-x^*\|^2
=\|(x-P_j(x))+(P_j(x)-x^*)\|^2
=\|x-P_j(x)\|^2+\|P_j(x)-x^*\|^2.
\end{align*}
Statement \eqref{imp:1} follows from the above identity
successively applied to $x$, $P_1(x)$, $P_2\circ P_1(x)$ etc.\ in lieu of $x$,
and from the definitions of $P$ and $r$.
Now the polarization identity yields
\begin{align*}
\|r(x)\|^2+\|x-P(x)\|^2
&=\|x-x^*\|^2-\|P(x)-x^*\|^2+\|x-P(x)\|^2\\
&=2(x-x^*)^T(x-P(x)).
\end{align*}
\end{proof}

The mapping $r$ behaves indeed like a residual.

\begin{lemma} \label{res:equiv}
The following statements are equivalent:
\begin{itemize}
\item [a)] We have $Ax=b$.
\item [b)] We have $P(x)=x$.
\item [c)] We have $r(x)=0$.
\end{itemize}
\end{lemma}

\begin{proof}
If statement a) holds, then 
\[P_j(x)=(I-\frac{a_ja_j^T}{\|a_j\|^2})x+\frac{b_j}{\|a_j\|^2}a_j=x,\quad 
j=1,\ldots,m,\]
which implies $P(x)=x$.

Assume that statement b) holds, and let $x^*\in\R^n$ be any point with $Ax^*=b$.
If $r(x)\neq 0$, then statement \eqref{imp:1} gives
\[\|P(x)-x^*\|^2<\|x-x^*\|^2,\]
which contradicts $P(x)=x$.
Hence statement c) holds.

If statement c) holds, then statement a) follows by induction.
We clearly have $a_1^Tx=b_1$.
If $a_i^Tx=b_i$ holds for $i=1,\ldots,j$, then $P_i(x)=x$ holds for $i=1,\ldots,j$,
and 
\[a_{j+1}^Tx-b_{j+1}=a_{j+1}^TP_j\circ\ldots\circ P_1(x)-b_{j+1}=r_{j+1}(x)=0.\]
By induction, we obtain $Ax=b$.
\end{proof}

\begin{remark}
A straight-forward modification of the convergence proof 
in \cite{Bregman} reveals that any sequence $(x_k)_{k\in\N}$ satisfying
\begin{equation}\label{better:than:K}
\|x_{k+1}-x^*\|^2\le\|P(x_k)-x^*\|^2\quad\forall\,k\in\N,\ x^*\in A^{-1}b,
\end{equation}
converges to a solution $x^*\in A^{-1}b$, and it is clear that the typical 
error estimates for cyclic projection-type methods as in Corollary 9.34 in
\cite{Deutsch} for the sequence $(P^k(x_0))_{k\in\N}$ also hold for the 
sequence $(x_k)_{k\in\N}$.
Hence we will focus on generating a sequence with the above property
\eqref{better:than:K} by minimizing the errors $\|x_{k+1}-x^*\|^2$ 
in affine search spaces at a relatively small computational cost.
\end{remark}

\section{Acceleration by line-search}

In this section, we recover the step-size from Theorem 4.1 in \cite{Tam}
with a straight-forward geometric argument.
In addition, we quantify the error reduction in terms 
of the difference between statements \eqref{imp:1} and
\eqref{gain:linesearch}.
All quantities involved in these formulas are known at runtime.

\begin{theorem}\label{greedy}
Let $x^*\in A^{-1}b$, and let $x\in\R^n$ with $P(x)\neq x$.
Then we have
\begin{align}
s^*:=\argmin_{s\in\R}\|(1-s)x+sP(x)-x^*\|^2
=\frac12+\frac{\|r(x)\|^2}{2\|P(x)-x\|^2},\nonumber\\
\|x-x^*\|^2-\|(1-s^*)x+s^*P(x)-x^*\|^2
=\frac{(\|r(x)\|^2+\|P(x)-x\|^2)^2}{4\|P(x)-x\|^2}.\label{gain:linesearch}
\end{align}
\end{theorem}

\begin{proof}
Using identity \eqref{imp:2}, we see that the strictly convex parabola
\begin{align*}
g(s)&=\|(1-s)x+sP(x)-x^*\|^2\\
&=\|x-x^*\|^2+2s(x-x^*)^T(P(x)-x)+s^2\|P(x)-x\|^2\\
&=\|x-x^*\|^2-s(\|r(x)\|^2+\|P(x)-x\|^2)+s^2\|P(x)-x\|^2
\end{align*}
has the unique minimum
\[s^*=\frac{\|r(x)\|^2+\|P(x)-x\|^2}{2\|P(x)-x\|^2}
=\frac12+\frac{\|r(x)\|^2}{2\|P(x)-x\|^2}.\]
The second statement follows from
\[g(s^*)=\|x-x^*\|^2-\frac{(\|r(x)\|^2+\|P(x)-x\|^2)^2}{4\|P(x)-x\|^2}.\]
\end{proof}

\begin{algorithm}\label{greedy:line:search}
\caption{Kaczmarz method with line-search,\\
originally proposed in \cite{Tam}\\
(complexity: $4\nnz(A)+3m+5n$ flops per cycle)}
\KwIn{
    $A\in\R^{m\times n}$,
    $b\in\R^m$,
    $x_0\in\R^n$   
}
\For{$k=0$ \KwTo $\infty$}{
    compute $P(x_k)$ and $r(x_k)$ in Kaczmarz cycle from $x_k$\;
    $d\gets P(x_k)-x_k$\;
    $\delta\gets\|d\|^2$\;
    \uIf{$\delta=0$}{\Return $x_k$\;}
    $\rho\gets\|r(x)\|^2$\;
    $s\gets\frac12+\frac{\rho}{2\delta}$\;
    $x_{k+1}\gets x_k+sd$\;
}
\end{algorithm}

It is a straight-forward consequence of Theorem \ref{greedy} 
that Algorithm \ref{greedy:line:search} is safe to use.

\begin{corollary}
Either Algorithm \ref{greedy:line:search} terminates in finite 
time and returns an iterate $x_k\in A^{-1}b$, or it generates
a well-defined sequence $(x_k)_k$ that has the property \eqref{better:than:K} 
and, for all $x^*\in A^{-1}b$, satisfies the identities
\begin{align}
&x_{k+1}=\argmin_{\xi\in\aff(x_k,P(x_k))}\|\xi-x^*\|^2,\label{ls:opt}\\
&\|x_k-x^*\|^2-\|x_{k+1}-x^*\|^2
=\frac{(\|r(x_k)\|^2+\|P(x_k)-x_k\|^2)^2}{4\|P(x_k)-x_k\|^2}.\label{concrete:decay:ls}
\end{align}
\end{corollary}

\begin{proof}
If Algorithm \ref{greedy:line:search} terminates after $k\in\N$ steps 
and returns an iterate $x_k\in\R^n$, then the stopping criterion implies
$P(x_k)=x_k$, and Lemma \ref{res:equiv} yields $Ax_k=b$.
Otherwise, all expressions in Algorithm \ref{greedy:line:search}
are well-defined.
By Theorem \ref{greedy}, formulas \eqref{ls:opt} and \eqref{concrete:decay:ls} hold,
and formula \eqref{ls:opt} implies \eqref{better:than:K}.
\end{proof}

\section{Acceleration by affine search}

It is possible to extend the above line-search to a search in an affine subspace
spanned by several previous iterates and the latest Kaczmarz cycle,
which improves the local error reduction.
Again, all required quantities and the exact reduction are computable at runtime.

\medskip

We begin by proving a simple geometric observation that will give meaning
to the stopping criterion of the accelerated iteration.

\begin{lemma} \label{no:stop}
Let $x^*\in A^{-1}b$, and let 
$x_1,\ldots,x_{\ell}\in\R^n$ be points such that the condition
\begin{equation}\label{last:opt}
x_\ell=\argmin_{\xi\in\aff(x_1,\ldots,x_\ell)}\|\xi-x^*\|^2
\end{equation}
holds.
If we have $P(x_\ell)\in\aff(x_1,\ldots,x_\ell)$, then we also have 
$\|r(x_\ell)\|^2=0$, $P(x_\ell)=x_\ell$ and $Ax_\ell=b$.
\end{lemma}

\begin{proof}
By statement \eqref{imp:1}, we have
\[\|P(x_\ell)-x^*\|^2=\|x_\ell-x^*\|^2-\|r(x_\ell)\|^2.\]
If we have $P(x_\ell)\in\aff(x_1,\ldots,x_\ell)$, then condition \eqref{last:opt}
yields $\|r(x_\ell)\|^2=0$, and Lemma \ref{res:equiv} implies that $Ax_\ell=b$.

\end{proof}

The following result provides a characterization of the minimizer 
\[\argmin_{\xi\in\aff(x_1,\ldots,x_\ell,P(x_\ell))}\|\xi-x^*\|^2\]
that does not use the unknown solution $x^*$ explicitly.
We formulate and prove this theorem for vectors indexed $x_1,\ldots,x_\ell$
to keep the notation simple, and we will use it later (see Corollary \ref{alg:2:cor})
for a varying number of vectors and varying indexation.

\begin{theorem} \label{affine:search:thm}
Let $x^*\in A^{-1}b$, let $x_1,\ldots,x_{\ell}\in\R^n$ 
be affinely independent points with \eqref{last:opt}
and $P(x_\ell)\notin\aff(x_1,\ldots,x_\ell)$.
Consider the matrices
\[V:=(x_1-x_\ell,\ldots,x_{\ell-1}-x_\ell)\in\R^{n\times(\ell-1)},\quad
M:=(V,P(x_\ell)-x_\ell)\in\R^{n\times\ell}\]
and define
\[\gamma:=\frac12(\|r(x_\ell)\|^2+\|P(x_\ell)-x_\ell\|^2).\]
Then the minimizer
\begin{equation} \label{min:s}
s^*:=\argmin_{s\in\R^\ell}\|x_\ell+Ms-x^*\|^2
\end{equation}
is the unique solution of the linear system
\begin{equation}\label{lgs}
M^TMs=\gamma e_\ell^\ell,
\end{equation}
where $e_\ell^\ell\in\R^\ell$ is the $\ell$-th unit vector, and we have 
\begin{equation}\label{better}
\|x_\ell-x^*\|^2-\|x_\ell+Ms^*-x^*\|^2
=\gamma s_\ell^*
=\gamma^2\frac{\det(V^TV)}{\det(M^TM)}.
\end{equation}
\end{theorem}

\begin{remark}
For an interpretation of the identity \eqref{better}, it is 
instructive to have a look at the case $\ell=2$.
Elementary computations show that whenever $P(x_2)\notin\aff(x_1,x_2)$,
the minimizer
\[s^*:=\argmin_{s\in\R^2}\|x_2+s_1(x_1-x_2)+s_2(P(x_2)-x_2)-x^*\|^2\]
satisfies
\begin{align*}
&\|x_2-x^*\|^2-\|x_2+s_1^*(x_1-x_2)+s_2^*(P(x_2)-x_2)-x^*\|^2\\
&=(1-\cos^2\angle(x_1-x_2,P(x_2)-x_2))^{-1}
\frac{(\|r(x_2)\|^2+\|P(x_2)-x_2\|^2)^2}{4\|P(x_2)-x_2\|^2}.
\end{align*}
Comparing with Theorem \ref{greedy}, we see that the planar search 
outperforms the line-search by a factor $(1-\cos^2\angle(x_1-x_2,P(x_2)-x_2))^{-1}$.
\end{remark}

\begin{proof}
Since the vectors $x_1,\ldots,x_\ell,P(x_\ell)$ are affinely independent,
the Gram\-ian matrices $V^TV$ and $M^TM$ are positive definite.
The first derivatives of the strictly convex quadratic function
\[g(s):=\|x_\ell+Ms-x^*\|^2\]
are given by
\begin{align*}
\frac{dg}{ds_j}(s)&=2\Big(x_\ell+Ms-x^*\Big)^T(x_j-x_\ell),\quad
j=1,\ldots,\ell-1,\\
\frac{dg}{ds_\ell}(s)&=2\Big(x_\ell+Ms-x^*\Big)^T(P(x_\ell)-x_\ell).
\end{align*}
Using statements \eqref{imp:2} and \eqref{last:opt}, we see that the unique minimizer
$s^*$ of $g$ solves the linear equations
\begin{align}
&(x_j-x_\ell)^TMs=\langle x^*-x_\ell,x_j-x_\ell\rangle=0,\quad
j=1,\ldots,\ell-1,\label{ortho}\\
&(P(x_\ell)-x_\ell)^TMs=\langle x^*-x_\ell,P(x_\ell)-x_\ell\rangle=\gamma,
\label{not:ortho}
\end{align}
which are subsumed in the linear system \eqref{lgs}.
Using Cramer's rule, we can express
\[s_\ell^*
=\frac{\det\begin{psmallmatrix}
\langle x_1-x_\ell,x_1-x_\ell\rangle&\ldots&\langle x_1-x_\ell,x_{\ell-1}-x_\ell\rangle&0\\
\vdots&&\vdots&\vdots\\
\langle x_{\ell-1}-x_\ell,x_1-x_\ell\rangle&\ldots&\langle x_{\ell-1}-x_\ell,x_{\ell-1}-x_\ell\rangle&0\\
\langle P(x_\ell)-x_\ell,x_1-x_\ell\rangle&\ldots&\langle P(x_\ell)-x_\ell,x_{\ell-1}-x_\ell\rangle&\gamma
\end{psmallmatrix}}
{\det(M^TM)}
=\gamma\frac{\det(V^TV)}{\det(M^TM)}.\]
From system \eqref{lgs}, we infer
\[\|Ms^*\|^2=(s^*)^TM^TMs^*=\gamma s_\ell^*,\]
and together with statements \eqref{ortho} and \eqref{not:ortho}, we conclude that
\begin{align*}
g(s^*)&=\|x_\ell+Ms^*-x^*\|^2
=\|x_\ell-x^*\|^2+2(s^*)^TM^T(x_\ell-x^*)+\|Ms^*\|^2\\
&=\|x_\ell-x^*\|^2+2s_\ell^*(P(x_\ell)-x_\ell)^T(x_\ell-x^*)+\gamma s_\ell^*
=\|x_\ell-x^*\|^2-\gamma s_\ell^*.
\end{align*}
\end{proof}

Algorithm \ref{basic:affine:search} uses Theorem \ref{affine:search:thm}
after every step to reduce the size of the error $\|x_{k+1}-x^*\|^2$
in statement \eqref{better:than:K}.

\begin{remark}\label{rem:affine}
a) The parameter $\ell\in\N_1$ in Algorithm \ref{basic:affine:search} controls 
how many of the previous iterates are used to span the affine search space.
When $\ell=1$, then Algorithm \ref{basic:affine:search} reduces to
Algorithm \ref{greedy:line:search}.
When $\ell>1$, the algorithm has a startup phase in which it grows the affine
basis of the search space, so that
\[V_0=[\ ],\quad
V_1=(x_0-x_1),\quad\ldots,\quad
V_{\ell-1}=(x_0-x_{\ell-1},\ldots,x_{\ell-2}-x_{\ell-1}).\]
After the startup phase, the algorithm keeps the latest $\ell$ iterates
and discards $x_{k-\ell}$, which gives
\[V_\ell=(x_1-x_\ell,\ldots,x_{\ell-1}-x_\ell),\quad
V_{\ell+1}=(x_2-x_{\ell+1},\ldots,x_\ell-x_{\ell+1}),\quad\ldots\]
When $\ell\ge n$, then the proof of Theorem \ref{alg:2:cor} reveals that
Algorithm \ref{basic:affine:search} terminates with an iterate $x_k\in A^{-1}b$, 
$k\le n$.

\medskip

b) The computational complexity of one step of Algorithm \ref{basic:affine:search} 
is composed in the following way:
\begin{enumerate}[(i)]
\item The Kaczmarz cycle requires $4\nnz(A)+m$ flops.
\item The computation of $d_k$, $\delta_k$ and $\rho_k$ requires $3n+2m$ flops.
\item Assembling $M$ requires $(k-j_k)n$ flops ($d_k$ is known).
\item Computing $M^TM$ requires $\frac12(k-j_k+1)(k-j_k+2)n$ flops.
\item Solving system \eqref{lgs} via LU factorization and forward and 
backward substitution requires 
$\frac23(k-j_k)^3+\frac72(k-j_k)^2+\frac56(k-j_k)$ flops.
\item Updating $x_k$ requires $2(k-j_k+1)n$ flops.
\end{enumerate}
After the startup phase, we have $k-j_k+1=\ell$, which gives a total computational complexity of roughly 
\[4\nnz(A)+(3+3\ell+\frac12\ell^2)n+3m+\ell^3.\]
The acceleration comes at a considerable cost, mostly caused by the assembly
and by solving system \eqref{lgs}, which is not desirable.
\end{remark}

\begin{algorithm}\label{basic:affine:search}
\caption{Kaczmarz method with affine search\\
(for complexity see Remark \ref{rem:affine} part b)}
\KwIn{
    $A\in\R^{m\times n}$,
    $b\in\R^m$,
    $x_0\in\R^n$,
    $\ell\in\N_1$    
}
\For{$k=0$ \KwTo $\infty$}{
    compute $P(x_k)$ and $r(x_k)$ in Kaczmarz cycle from $x_k$\;
    $d_k\gets P(x_k)-x_k$\;
    $\delta_k\gets\|d_k\|^2$\;
    \uIf{$\delta=0$}{\Return $x_k$\;}
    $\rho_k\gets\|r(x_k)\|^2$\;
    $\gamma_k\gets\frac12(\rho_k+\delta_k)$\;
    $j_k\gets\max\{k-\ell+1,0\}$\;
    $V_k\gets(x_{j_k}-x_k,\ldots,x_{k-1}-x_k)\in\R^{n\times(k-j_k)}$\;
    assemble $M_k^TM_k$ from $M_k=(V_k,d_k)\R^{n\times(k-j_k+1)}$\;
    solve $M_k^TM_ks_k=\gamma_k e_{k-j_k+1}^{k-j_k+1}$ for $s_k$\;
    $x_{k+1}\gets x_k+M_ks_k$\;
}
\end{algorithm}

\begin{theorem} \label{alg:2:cor}
Let $x^*\in A^{-1}b$.
Either Algorithm \ref{basic:affine:search} terminates and returns an iterate 
$x_k\in A^{-1}b$, or it generates
a well-defined sequence $(x_k)_k$ that satisfies
\begin{align}
&x_{k+1}=\argmin_{\xi\in\aff(x_{j_k},\ldots,x_k,P(x_k))}\|\xi-x^*\|^2,\label{cor:a}\\
&\|x_k-x^*\|^2-\|x_{k+1}-x^*\|^2=\gamma_k^2\frac{\det(V_k^TV_k)}{\det(M_k^TM_k)}
\label{cor:b}
\end{align}
for all $k\in\N$. 
In particular, it has the property \eqref{better:than:K}.
\end{theorem}

\begin{proof}
We prove by induction that Algorithm \ref{basic:affine:search} either
returns a solution in finite time or generates a sequence such that 
identities \eqref{cor:a} and \eqref{cor:b} and the following statements 
hold for every $k\in\N$:
\begin{itemize}
\item [a)] The vectors $x_{j_k},\ldots,x_k$ are affinely independent.
\item [b)] We have $x_k=\argmin_{\xi\in\aff(x_{j_k},\ldots,x_k)}\|\xi-x^*\|^2$.
\item [c)] We have $P(x_k)\notin\aff(x_{j_k},\ldots,x_k)$.
\end{itemize}
If $k=0$, then properties a) and b) are trivially satisfied.

\medskip

Now assume that Algorithm \ref{basic:affine:search} has generated iterates 
$x_0,\ldots,x_k\in\R^n$ with properties a), b).

\medskip

If Algorithm \ref{basic:affine:search} terminates and returns $x_k$, 
then the stopping criterion implies that $P(x_k)=x_k$, and
Lemma \ref{res:equiv} implies $Ax_k=b$.

\medskip
If Algorithm \ref{basic:affine:search} does not terminate, then we have 
$P(x_k)\neq x_k$.
Because of statement b), Lemma \ref{no:stop} implies statement c).
Since statements a), b) and c) hold for $k$, the vectors $x_{j_k},\ldots,x_k$
satisfy all assumptions of Theorem \ref{affine:search:thm}.
The linear system \eqref{lgs} with the matrix
\[M=(x_{j_k}-x_k,\ldots,x_{k-1}-x_k,P(x_k)-x_k)\]
possesses a unique solution $s^*\in\R^{k-{j_k}+1}$, and the iterate
$x_{k+1}:=x_k+Ms^*$ is well-defined.
Because of statements a) and c), we have 
\[\det(V^TV)>0\quad\text{and}\quad\det(M^TM)>0,\]
so by statement \eqref{better}, we have $s_{k-{j_k}+1}^*\neq 0$.
Combining this fact with statements a) and c) yields statement a) 
with $k+1$ in lieu of $k$.
Statement \eqref{min:s} implies statement \eqref{cor:a} and, since
\[\aff(x_{j_{k+1}},\ldots,x_k,x_{k+1})\subset\aff(x_{j_k},\ldots,x_k,P(x_k)),\] 
also statement b) for $k+1$ in lieu of $k$.
In addition, statement \eqref{better} implies the identity \eqref{cor:b}.

\bigskip

\end{proof}

\section{Efficient updating}

The goal of this section is to simplify the solution of the linear system \eqref{lgs},
which must be solved after every Kaczmarz cycle to determine $x_{k+1}$ from
the previous iterates and the vector $P(x_k)$.

\medskip

Our first result shows that updating the submatrix $V^TV$ of the matrix
$M^TM$ from one iteration to another is straight-forward.

\begin{lemma} \label{update}
In the situation of Theorem \ref{affine:search:thm}, 
and denoting $x_{\ell+1}:=x_\ell+Ms^*$, we have
\begin{align*}
&\langle x_i-x_{\ell+1},x_j-x_{\ell+1}\rangle=(V^TV)_{ij}+\gamma s_\ell^*,
&&1\le i,j\le\ell-1,\\
&\langle x_i-x_{\ell+1},x_{\ell}-x_{\ell+1}\rangle=\gamma s_\ell^*,
&&1\le i\le\ell.
\end{align*}
\end{lemma}

\begin{proof}
We can express
\[x_i-x_\ell=Me_i,\quad i=1,\ldots,\ell-1.\]
For $1\le i,j\le\ell-1$, we use the identity \eqref{lgs} to obtain
\begin{align*}
&\langle x_i-x_{\ell+1},x_j-x_{\ell+1}\rangle
=\langle x_i-x_\ell-Ms^*,x_j-x_\ell-Ms^*\rangle\\
&=\langle Me_i-Ms^*,Me_j-Ms^*\rangle\\
&=e_i^TM^TMe_j-e_i^TM^TMs^*-e_j^TM^TMs^*+(s^*)^TM^TMs^*\\
&=(M^TM)_{ij}+\gamma s_\ell^*
=(V^TV)_{ij}+\gamma s_\ell^*
\end{align*}
For $1\le i<\ell$, we compute
\begin{align*}
&\langle x_i-x_{\ell+1},x_\ell-x_{\ell+1}\rangle
=\langle x_i-x_\ell-Ms^*,-Ms^*\rangle\\
&=\langle Me_i-Ms^*,-Ms^*\rangle
=-e_i^TM^TMs^*+(s^*)^TM^TMs^*=\gamma s_\ell^*,
\end{align*}
and we also obtain
\[\langle x_\ell-x_{\ell+1},x_\ell-x_{\ell+1}\rangle
=(s^*)^TM^TMs^*
=\gamma s_\ell^*.\]
\end{proof}

We will see (in the proof of Theorem \ref{algo:3:works}) that the matrices $V^TV$
generated by Algorithms \ref{basic:affine:search} and \ref{greedy:affine:search} 
have the structure of the matrix $B$ defined below with known coefficients $\alpha_j$.

\begin{lemma} \label{ringed}
Let $\alpha\in\R^n$, and let 
$f_1^n,\ldots,f_n^n\in\R^n$ be given by $f_j^n=\sum_{i=1}^je_i^n$,
where $e_i^n\in\R^n$ denotes the $i$-th unit vector.
Then the matrix
\[B:=\sum_{j=1}^n\alpha_jf_j^n(f_j^n)^T\]
has the structure
\[B=\begin{pmatrix}
\sum_{i=1}^n\alpha_i&\sum_{i=2}^n\alpha_i&\cdots&\sum_{i=n}^n\alpha_i\\
\sum_{i=2}^n\alpha_i&\sum_{i=2}^n\alpha_i&&\vdots\\
\vdots&&\ddots&\vdots\\
\sum_{i=n}^n\alpha_i&\cdots&\cdots&\sum_{i=n}^n\alpha_i
\end{pmatrix}.\]
If $\alpha_j\neq 0$ for $j=1,\ldots,n$, then
\[C:=\begin{pmatrix}
\alpha_1^{-1}&-\alpha_1^{-1}\\
-\alpha_1^{-1}&\alpha_1^{-1}+\alpha_2^{-1}&-\alpha_2^{-1}\\
&-\alpha_2^{-1}&\alpha_2^{-1}+\alpha_3^{-1}&-\alpha_3^{-1}\\
&&\ddots&\ddots&\ddots\\
&&&-\alpha_{n-2}^{-1}&\alpha_{n-2}^{-1}+\alpha_{n-1}^{-1}&-\alpha_{n-1}^{-1}\\
&&&&-\alpha_{n-1}^{-1}&\alpha_{n-1}^{-1}+\alpha_{n}^{-1}
\end{pmatrix}\]
is the inverse of the matrix $B$.
\end{lemma}

\begin{proof}
When $\alpha_j\neq 0$ for $j=1,\ldots,n$, then the matrix $C$ is well-defined.
Multiplying the matrices $B$ and $C$ yields the identity.
\end{proof}

In conjunction with Lemmas \ref{update} and \ref{ringed}, the next lemma shows 
that the linear system \eqref{lgs} can be solved in linear time.

\begin{lemma} \label{rank:2}
Let $B\in\R^{n\times n}$ be invertible, 
let $p\in\R^n$ and let $\delta>0$ and $\gamma\in\R$.
If the matrix 
\[G:=\begin{pmatrix}B&p\\p^T&\delta\end{pmatrix}\]
is invertible, then we have $p^TB^{-1}p\neq\delta$, and the solution 
of the linear system
\[Gx=\gamma e^n_n\]
is given by 
\[G^{-1}\gamma e^n_n
=\frac{\gamma}{p^TB^{-1}p-\delta}\begin{pmatrix}B^{-1}p\\-1\end{pmatrix}.\]
\end{lemma}

\begin{proof}
Since $G$ is nonsingular, and since
\[\begin{pmatrix}B&p\\p^T&\delta\end{pmatrix}\begin{pmatrix}
B^{-1}p\\-1\end{pmatrix}=\begin{pmatrix}0\\p^TB^{-1}p-\delta\end{pmatrix}\]
holds, we have $p^TCp\neq\delta$, and the desired result follows.
\end{proof}

Lemmas \ref{update}, \ref{ringed} and \ref{rank:2} inspire Algorithm
\ref{greedy:affine:search}.
We require $\ell\ge 2$, because for $\ell=1$, when Algorithm \ref{basic:affine:search}
reduces to Algorithm \ref{greedy:line:search}, there is no data to be updated.
By $C(\alpha)$, we denote the matrix $C$ from Lemma \ref{ringed}
given by the parameter vector $\alpha$ and its dimension.

In the initial step of Algorithm \ref{greedy:affine:search},
the matrices $V_0$ and $C_0$ as well as the vectors $p_0$ and $q_0$
are empty and have to be ignored where they occur.
We split the solution $s_k$ of the linear system
$M_k^TM_ks_k=\gamma_ke_{k-j+1}$ into the vector $\overline{s_k}\in\R^{k-j}$ 
of the first several components and the last component $\underline{s_k}\in\R$
to exploit the structure of system \eqref{lgs}.

\begin{algorithm}\label{greedy:affine:search}
\caption{Kaczmarz method with enhanced affine search\\
(for complexity see Remark \ref{discuss:enhanced})}
\KwIn{
    $A\in\R^{m\times n}$,
    $b\in\R^m$,
    $x_0\in\R^n$,
    $\ell\in\N_2$     
}
\For{$k=0$ \KwTo $\infty$}{
    compute $P(x_k)$ and $r(x_k)$ in Kaczmarz cycle from $x_k$\;
    $d_k\gets P(x_k)-x_k$\;
    $\delta_k\gets\|d_k\|^2$\;
    \uIf{$\delta_k=0$}{\Return $x_k$\;}    
    $\rho_k\gets\|r(x_k)\|^2$\;
    $\gamma_k\gets\frac12(\rho_k+\delta_k)$\;
    $j_k\gets\max\{k-\ell+1,0\}$\;    
    $V_k\gets(x_{j_k}-x_k,\ldots,x_{k-1}-x_k)\in\R^{n\times(k-{j_k})}$\;
    $p_k\gets V_k^Td_k\in\R^{k-{j_k}}$\;
    $C_k\gets C(\gamma_{j_k}\underline{s_{j_k}},\ldots,\gamma_{k-1}\underline{s_{k-1}})
        \in\R^{(k-{j_k})\times(k-{j_k})}$\;
    $q_k\gets C_kp_k
        \in\R^{k-{j_k}}$\;
    $\underline{s_k}\gets\frac{\gamma_k}{\delta-p_k^Tq_k}\in\R$\;
    $\overline{s_k}\gets-\underline{s_k}q_k\in\R^{k-{j_k}}$\;
    $x_{k+1}\gets x_k+V_k\overline{s_k}+\underline{s_k}d_k$\;    
}
\end{algorithm}

\begin{theorem} \label{algo:3:works}
Algorithms \ref{basic:affine:search} and \ref{greedy:affine:search} generate 
identical iterations.
\end{theorem}


\begin{proof}
We prove by induction that one of the following alternatives holds:
\begin{itemize}
\item [i)] Algorithms \ref{basic:affine:search} 
and \ref{greedy:affine:search} both terminate in step $k$.
\item [ii)] We have 
\[V_k^TV_k=\begin{cases}\sum_{i=1}^{k-j_k}\gamma_{j_k+i-1}\underline{s_{j_k+i-1}}f_i^{k-j_k}(f_i^{k-j_k})^T,&k>0\\\mathrm{the\ empty\ matrix}\ [\ ],&k=0,\end{cases}\]
where $f_1^{k-j_k},\ldots,f_{k-j_k}^{k-j_k}\in\R^{k-j_k}$ 
are given by $f_i^{k-j_k}=\sum_{h=1}^ie_h^{k-j_k}$,
and both algorithms compute identical $\gamma_k$, $s_k$ and $x_{k+1}$.
\end{itemize}

When $k=0$, both algorithms terminate if and only if $d_0=P(x_0)-x_0=0$.
Otherwise, both algorithms compute identical $\delta_0$, $\rho_0$ and $\gamma_0$,
and they both have $j_0=0$ and $V_0=[\ ]$.
Algorithm \ref{basic:affine:search} solves
\begin{equation}\label{l1}
\delta_0s_0=\|d_0\|^2s_0=M_0^TM_0s_0=\gamma_0,
\end{equation}
while Algorithm \ref{greedy:affine:search} has $p_0=[\ ]$, $C_0=[\ ]$ and $q_0=[\ ]$,
computes $\underline{s_0}=\delta_0^{-1}\gamma_0$ and sets $\overline{s_0}=[\ ]$.
Hence both algorithms generate identical $s_0\in\R$ and the same next iterate
\begin{equation}\label{l2}
x_1=x_0+s_0d_0\in\R^n.
\end{equation}

\medskip

Now assume that alternative ii) holds for $0,\ldots,k$.
Then both algorithms compute identical $d_{k+1}$ and $\delta_{k+1}$, and both 
terminate and return $x_{k+1}$ if and only if $\delta_{k+1}=0$.
Otherwise, they compute identical $\rho_{k+1}$ and $\gamma_{k+1}$.
For the update $V_k\mapsto V_{k+1}$, we distinguish the following cases:
\begin{itemize}
\item [a)] When $k=0$, we have $j_0=j_1=0$ and $V_1=(x_0-x_1)$,
so using statements \eqref{l1} and \eqref{l2}, we find 
\[V_1^TV_1=\|x_0-x_1\|^2=s_0^2\delta_0=\gamma_0s_0.\]
\item [b)] When $k>0$ and $j_{k+1}=j_k$, then we have $j_{k+1}=j_k=0$ as well as
$V_k=(x_0-x_k,\ldots,x_{k-1}-x_k)$ and 
$V_{k+1}=(x_0-x_{k+1},\ldots,x_k-x_{k+1})$. 
Lemma \ref{update} tells us that
\begin{align*}
&\langle x_i-x_{k+1},x_j-x_{k+1}\rangle=(V_k^TV_k)_{ij}+\gamma_k\underline{s_k},
&&0\le i,j\le k-1,\\
&\langle x_i-x_{k+1},x_k-x_{k+1}\rangle=\gamma_k\underline{s_k},
&&0\le i\le k.
\end{align*}
The induction hypothesis implies that
\begin{align*}
V_{k+1}^TV_{k+1}
&=\begin{pmatrix}V_k^TV_k&0\\0^T&0\end{pmatrix}
+\gamma_k\underline{s_k}\mathbbm{1}_{\R^{k+1}}\mathbbm{1}_{\R^{k+1}}^T\\
&=\sum_{i=1}^{k-j_k}\gamma_{j_k+i-1}\underline{s_{j_k+i-1}}
\begin{pmatrix}f_i^{k-j_k}\\0\end{pmatrix}((f_i^{k-j_k})^T,0)
+\gamma_k\underline{s_k}\mathbbm{1}_{\R^{k+1}}\mathbbm{1}_{\R^{k+1}}^T\\
&=\sum_{i=1}^{k+1-j_{k+1}}\gamma_{j_{k+1}+i-1}\underline{s_{j_k+i-1}}
f_i^{k+1-j_{k+1}}(f_i^{k+1-j_{k+1}})^T,
\end{align*}
where we have used that $\gamma_k=\gamma_{k+1}=0$ and
$\mathbbm{1}_{\R^{k+1}}=f_{k+1}^{k+1}$.
\item [c)] When $k>0$ and $j_{k+1}\neq j_k$, then we have $j_{k+1}=k-\ell+2$
and $j_k=k-\ell+1$, and we consider
$V_k=(x_{k-\ell+1}-x_k,\ldots,x_{k-1}-x_k)$ and 
$V_{k+1}=(x_{k-\ell+2}-x_{k+1},\ldots,x_k-x_{k+1})$. 
Again by Lemma \ref{update}, and defining $\mu\in\R$ and $v\in\R^{\ell-1}$ by
\begin{align*}
&\mu:=\langle x_{k-\ell+1}-x_{k+1},x_{k-\ell+1}-x_{k+1}\rangle\\
&v_i:=\langle x_{k-\ell+1}-x_{k+1},x_{k-\ell+i+1}-x_{k+1}\rangle,\quad
i=1,\ldots,\ell-1,
\end{align*}
we find for similar reasons that
\begin{align}
&\begin{pmatrix}\mu&v^T\\v&V_{k+1}^TV_{k+1}\end{pmatrix}
=\begin{pmatrix}V_k^TV_k&0\\0^T&0\end{pmatrix}
+\gamma_k\underline{s_k}\mathbbm{1}_{\R^{\ell}}\mathbbm{1}_{\R^{\ell}}^T\nonumber\\
&=\sum_{i=1}^{k-j_k}\gamma_{j_k+i-1}\underline{s_{j_k+i-1}}
\begin{pmatrix}f_i^{k-j_k}\\0\end{pmatrix}((f_i^{k-j_k})^T,0)
+\gamma_k\underline{s_k}\mathbbm{1}_{\R^{\ell}}\mathbbm{1}_{\R^{\ell}}^T\label{have}\\
&=\sum_{i=1}^{k+1-j_{k}}\gamma_{j_{k}+i-1}\underline{s_{j_k+i-1}}
f_i^{k+1-j_k}(f_i^{k+1-j_k})^T.\nonumber
\end{align}
The desired statement
\begin{equation}\label{want}
V_{k+1}^TV_{k+1}=\sum_{i=1}^{k+1-j_{k+1}}\gamma_{j_{k+1}+i-1}\underline{s_{j_{k+1}+i-1}}f_i^{k+1-j_{k+1}}(f_i^{k+1-j_{k+1}})^T
\end{equation}
can be verified by a component-wise comparison of equations \eqref{have} and
\eqref{want}.
\end{itemize}
Hence in all three cases, the matrix $V_{k+1}^TV_{k+1}$ has the desired representation,
and Lemma \ref{ringed} yields that 
\[(V_{k+1}^TV_{k+1})^{-1}
=C(\gamma_{j_{k+1}}\underline{s_{j_{k+1}}},\ldots,\gamma_k\underline{s_k})
=C_{k+1}.\]
Since we have
\[M_{k+1}^TM_{k+1}=\begin{pmatrix}
V_{k+1}^TV_{k+1}&V_{k+1}^Td_{k+1}\\d_{k+1}V_{k+1}&\delta
\end{pmatrix},\]
Lemma \ref{rank:2} tells us that the remaining steps of Algorithm
\ref{greedy:affine:search} compute the solution 
$s_k=(\overline{s_k}^T,\underline{s_k})^T$ of the linear system 
\[M_{k+1}^TM_{k+1}s_{k+1}=\gamma_{k+1}e^{k+1-j_{k+1}}_{k+1-j_{k+1}}\]
and hence the same iterate $x_{k+1}$ as Algorithm \ref{basic:affine:search}.
\end{proof}

\begin{remark}\label{discuss:enhanced}
The computational complexity of one step of Algorithm \ref{greedy:affine:search}
is composed in the following way:
\begin{enumerate}[(i)]
\item The Kaczmarz cycle requires $4\nnz(A)+m$ flops.
\item The computation of $d_k$, $\delta_k$ and $\rho_k$ requires $3n+2m$ flops.
\item Assembling $V_k$ requires $(k-j_k)n$ flops.
\item Computing $p_k$ requires $2(k-j_k)n$ flops.
\item Computing $q_k$ requires $4(k-j_k)$ flops.
\item Computing $\overline{s_k}$ requires $(k-j_k)$ flops.
\item Updating $x_k$ requires $2(k-j_k+1)n$ flops.
\end{enumerate}
After the startup phase, when $k-j_k+1=\ell$, we have a total computational complexity
of roughly 
\[4\nnz(A)+(3+5\ell)n+3m+5\ell.\]
Hence Algorithm \ref{basic:affine:search} arrives at the same numerical 
results as Algorithm \ref{greedy:affine:search}, but it replaces the 
most expensive operations (assembly of $M^TM$ at roughly $\frac12\ell^2n$ flops
and solution of system \eqref{lgs} at roughly $\ell^3$ flops) with 
cheap ones (matrix-vector product with tridiagonal matrix at roughly $4\ell$ flops
and inner product at roughly $\ell$ flops).
\end{remark}

\section{Application to the random Kaczmarz method}

The random Kaczmarz method is displayed in Algorithm \ref{RK}.
We organize the iterations in epochs of $m$ projection steps.
Please refer to Algorithm \ref{RK} for the random indices used in this section.

\medskip

\begin{algorithm}\label{RK}
\caption{Random Kaczmarz method,\\
originally proposed in \cite{Strohmer}\\
(complexity: $4\nnz(A)+m$ flops per epoch with $m$ projections)}
\KwIn{
    $A\in\R^{m\times n}$,
    $b\in\R^m$,
    $x_0\in\R^n$   
}
\For{$k=0$ \KwTo $\infty$}{
    $x_{k,0}\gets x_k$\;
    \For{$j=0$ \KwTo $m-1$}{
        draw $i_{k,j}$ uniformly at random from $\{1,\ldots.m\}$\;
        $x_{k,j+1}\gets P_{i_{k,j}}(x_{k,j})$\;
    }
    $x_{k+1}\gets x_{x,m}$\;
}
\end{algorithm}

\begin{algorithm}\label{RK:accel}
\caption{Random Kaczmarz method with affine search\\
(for an upper bound on the complexity see Remark \ref{discuss:enhanced})}
\KwIn{
    $A\in\R^{m\times n}$,
    $b\in\R^m$,
    $x_0\in\R^n$,
    $\ell\in\N_2$
}
\For{$k=0$ \KwTo $\infty$}{
	\Repeat{$x_{k,m}\neq x_k$}{
	    $\rho_k\gets 0$\;
        $x_{k,0}\gets x_k$\;
        \For{$j=0$ \KwTo $m-1$}{
            draw $i_{k,j}$ uniformly at random from $\{1,\ldots.m\}$\;
            $\rho_k\gets\rho_k+(a_{i_{k,j}}^Tx_{k,j}-b_{i_{k,j}})^2/\|a_{i_{k,j}}\|^2$\;
            $x_{k,j+1}\gets P_{i_{k,j}}(x_{k,j})$\;
        }
    }
    $d_k\gets x_{k,m}-x_k$\;
    $\delta_k\gets\|d_k\|^2$\;
    $\gamma_k\gets\frac12(\rho_k+\delta_k)$\;
    $j_k\gets\max\{k-\ell+1,0\}$\;    
    $V_k\gets(x_{j_k}-x_k,\ldots,x_{k-1}-x_k)\in\R^{n\times(k-{j_k})}$\;
    $p_k\gets V_k^Td_k\in\R^{k-{j_k}}$\;
    $C_k\gets C(\gamma_{j_k}\underline{s_{j_k}},\ldots,\gamma_{k-1}\underline{s_{k-1}})
        \in\R^{(k-{j_k})\times(k-{j_k})}$\;
    $q_k\gets C_kp_k
        \in\R^{k-{j_k}}$\;
    $\underline{s_k}\gets\frac{\gamma_k}{p_k^Tq_k-\delta_k}\in\R$\;
    $\overline{s_k}\gets-\underline{s_k}q_k\in\R^{k-{j_k}}$\;
    $x_{k+1}\gets x_k+V_k\overline{s_k}+\underline{s_k}d_k$\;    
}
\end{algorithm}

The random Kaczmarz method is known to converge in expectation
when $A$ has full rank and $x^*:=A^{-1}b$ is unique, see Theorem 2 in \cite{Strohmer}.
The final statement of the induction step in its proof shows that
\begin{equation}\label{stro}
\mathbbm{E}\|x_{k+1}-x^*\|^2\le(1-\kappa(A)^{-2})^m\mathbbm{E}\|x_k-x^*\|^2
\quad\forall\,k\in\N,
\end{equation}
where $\kappa(A)>0$ is a specific condition number, see Section 1 of \cite{Strohmer}
for details. 
Statement \eqref{stro} reveals that an acceleration of the sequence $(x_k)_k$
that reduces the error $\|x_k-x^*\|^2$ maintains the convergence properties of the 
sequence as well as the error estimate.
This motivates us to transfer the acceleration techniques from the previous
sections to the random Kaczmarz method, which gives Algorithm \ref{RK:accel}.

\begin{remark}\label{rK:remark}
Algorithm \ref{RK:accel} is justified by the following reasoning:
The $k$-th epoch of the random Kaczmarz method can be regarded as one cycle
\[\tilde{P}^{(k)}(x):=(P_{i_{k,m-1}}\circ\ldots\circ P_{i_{k,0}})(x)\]
of the deterministic Kaczmarz method applied to the matrix 
$\tilde{A}^{(k)}\in\R^{m\times n}$
and a vector $\tilde{b}^{(k)}\in\R^m$ given by
\[\tilde{A}^{(k)}:=(a_{i_{k,0}},\ldots,a_{i_{k,m-1}})^T,\quad
\tilde{b}^{(k)}:=(b_{i_{k,0}},\ldots,b_{i_{k,m-1}})^T,\]
which gives rise to the residual
\[\tilde{r}^{(k)}(x):=\begin{pmatrix}
(a_{i_{k,0}}^Tx-b_{i_{k,0}})/\|a_{i_{k,0}}\|\\
(a_{i_{k,1}}^TP_{i_{k,0}}(x)-b_{i_{k,1}})/\|a_{i_{k,1}}\|\\
\vdots\\
(a_{i_{k,m-1}}^TP_{i_{k,m-2}}\circ\ldots\circ P_{i_{k,0}}(x)-b_{i_{k,m-1}})/\|a_{i_{k,m-1}}\|
\end{pmatrix}.\]
The solution $x^*=A^{-1}b$ also solves $\tilde{A}^{(k)}x^*=\tilde{b}^{(k)}$,
and all statements on errors given in the previous sections remain valid.

\medskip

However, it is no longer true that $\tilde{P}^{(k)}(x_k)\in\aff(x_{j_k},\ldots,x_k)$ 
if and only if $Ax_k=b$, because $\tilde{A}^{(k)}$ 
and $\tilde{b}^{(k)}$ are only subsamples of $A$ and $b$, which invalidates 
the stopping criteria we previously used.
On the other hand, the inclusion $\tilde{P}^{(k)}(x_k)\in\aff(x_{j_k},\ldots,x_k)$ 
still implies that $\tilde{P}^{(k)}(x_k)=x_k$ and $\tilde{r}^{(k)}(x_k)=0$.
Hence, in this situation, we can ignore the last epoch in the acceleration scheme
(see line 10 of Algorithm \ref{RK:accel}), and accelerate only when the random
Kaczmarz method made progress.
This guarantees that whenever Theorem \ref{affine:search:thm} is invoked,
its assumptions are satisfied.
\end{remark}

\begin{figure}[t]\begin{center}
\includegraphics[scale=0.8]{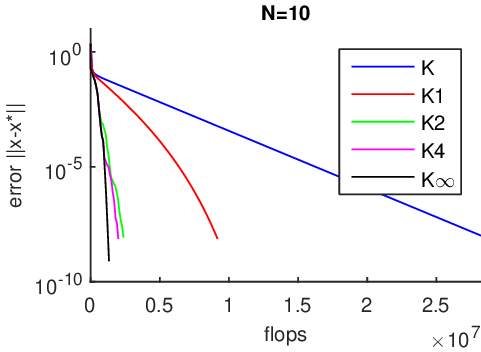}\hfill
\includegraphics[scale=0.8]{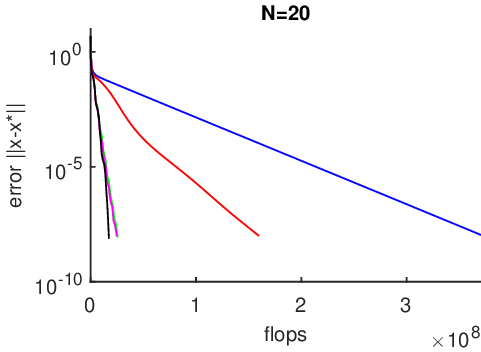}\\
\includegraphics[scale=0.8]{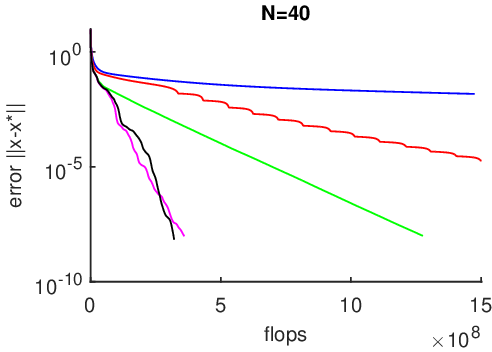}
\caption{Algorithm \ref{greedy:affine:search} in CT example
with $N\times N$ pixels, $N=10,20,40$.
Notation: 
K - Kaczmarz method, 
K1 - Kaczmarz method with line-search,
K$\ell$ - Kaczmarz method with at most $\ell$-dimensional affine search space.
\label{uK}}
\end{center}\end{figure}

\begin{figure}[t]\begin{center}
\includegraphics[scale=0.8]{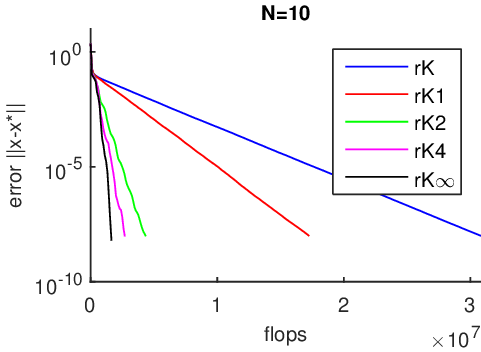}\hfill
\includegraphics[scale=0.8]{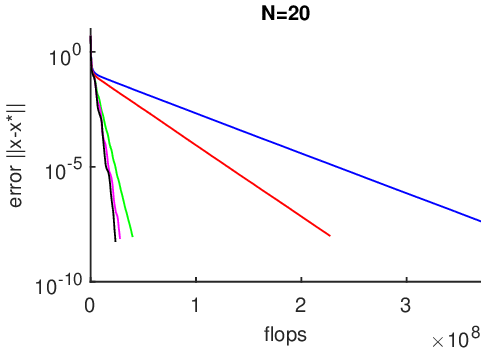}\\
\includegraphics[scale=0.8]{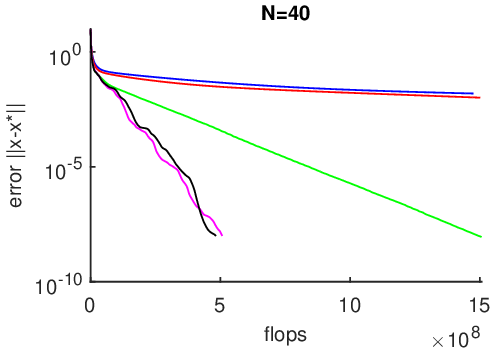}
\caption{Algorithm \ref{RK:accel} in CT example
with $N\times N$ pixels, $N=10,20,40$.
Notation: 
rK - random Kaczmarz method, 
K1 - random Kaczmarz method with line-search,
K$\ell$ - random Kaczmarz method with at most $\ell$-dimensional affine search space.\label{rK}}
\end{center}\end{figure}

\section{Numerical results}

We test the algorithms presented in this paper in the context of the 
computerized tomography problem, which is one of the most important 
applications of the Kaczmarz method. 
To generate benchmark problems, we apply the \texttt{paralleltomo} function 
of the AIR Tools library \cite{Hansen} with default parameters to the 
Shepp-Logan medical phantom.

\medskip

In all simulations, we choose the initial guess $x_0=0$ and apply an initial
random shuffling to the rows of $A$, because the canonical ordering tends to
induce very slow and hence atypical convergence.
To see how the method behaves under scaling, we investigate the following 
scenarios, where
\[\mathrm{oncost}(\ell):=
\frac{\mathrm{cost}(\mathrm{acceleration}(\ell))}
{\mathrm{cost}(\mathrm{Kaczmarz\ cycle})}
\approx\frac{(3+5\ell)n+2m+5\ell}{4\nnz(A)+m}\]
measures the cost of an acceleration step relative in terms of the cost 
of a Kaczmarz cycle:
\begin{itemize}
\item [a)] object resolution 10x10, 
process matrix $A\in\R^{2296\times 100}$, 
number of nonzero elements $\mathrm{nnz}(A)=22820$, 
sparsity $\frac{\mathrm{nnz}(A)}{\#\mathrm{entries}(A)}\approx 0.1$,
condition number $\mathrm{cond}(A)\approx 62$,
$\mathrm{oncost}(\ell)\approx 0.052+0.005\ell$.
\item [b)] object resolution 20x20,
process matrix $A\in\R^{4584\times 400}$, 
number of nonzero elements $\mathrm{nnz}(A)=91608$, 
sparsity $\frac{\mathrm{nnz}(A)}{\#\mathrm{entries}(A)}=0.05$, 
condition number $\mathrm{cond}(A)\approx 114$,
$\mathrm{oncost}(\ell)\approx 0.028+0.005\ell$.
\item [c)] object resolution 40x40,
process matrix $A\in\R^{9178\times 1600}$, 
number of nonzero elements $\mathrm{nnz}(A)=366496$, 
sparsity $\frac{\mathrm{nnz}(A)}{\#\mathrm{entries}(A)}=0.025$, 
condition number $\mathrm{cond}(A)\approx 480$,
$\mathrm{oncost}(\ell)\approx 0.016+0.005\ell$.
\end{itemize}
We see that in all three scenarios, the cost of an acceleration step relative
to the cost of a Kaczmarz cycle is small, though the matrices are moderately
sparse.

\medskip

It is well-known that the normal equations \eqref{lgs} are prone to become
ill-conditioned.
We observed unstable behavior of Algorithm \ref{basic:affine:search} e.g.\ 
in scenario a) at $\ell=20$.
Though Algorithm \ref{greedy:affine:search} essentially solves the same problem,
it remained stable, which is probably a benefit of applying an explicitly known 
inverse over solving the linear system numerically. 
Only when the approximation error $\|x_k-x^*\|$ was very small (roughly $10^{-13}$),
we observed instability in the form of oscillating errors.

Algorithm \ref{greedy:affine:search} always performs better than Algorithm
\ref{basic:affine:search}, and the outperformance increases with the problem dimension 
$n$ and the dimension $\ell$ of the affine search space, see Remark \ref{rem:affine}b
and Remark \ref{discuss:enhanced}.
As it is also more stable, we only display the numerical errors of Algorithm
\ref{greedy:affine:search} in Figure \ref{uK}.

\medskip

The wobble that is most pronounced in the error plot of $K_1$ for $N=40$ is neither
caused by an unstable algorithm nor an artefact.
It is typical across a range of acceleration schemes (not presented in this 
paper, but investigated by the author numerically) and seems to be caused 
by going back and forth between the Euclidean geometry and the geometry of the
Kaczmarz map $P$.

Roughly speaking, the error curves of the accelerated Kaczmarz methods
cluster at the error curve of $K_\infty$, which is the variant of 
Algorithm \ref{greedy:affine:search} that spans the affine search space using 
all previously computed iterates in every step.
This seems to suggest that the benefit of working with many or all previous iterates
outweighs the additional cost incurred by processing them.

\medskip

Figure \ref{rK} shows the performance of Algorithm \ref{RK:accel}, and it
demonstrates that our acceleration technique can be successfully applied 
to the random Kaczmarz method.
Comparing figures \ref{uK} and \ref{rK} in terms of absolute values is not
meaningful, because the performance of the deterministic Kaczmarz method depends 
on the chosen order of the rows of $A$, and the performance of the random 
Kaczmarz method depends to some degree on the particular random numbers drawn.

However, there seems to be a slight qualitative difference in the performance
of the accelerated methods between the deterministic and the random setting.
In the random setting, a larger $\ell$ seems to be needed to achieve a similar 
level of outperformance of the plain Kaczmarz method as in the deterministic setting, which is particularly noticeable in the error plot of the methods $K_1$ and $rK_1$ 
for $N=40$. 
On the other hand, in both settings, the method $rK_\infty$ always clearly outperforms
the Kaczmarz method.

\bibliographystyle{plain}
\bibliography{gGKa}

\end{document}